\documentclass[11pt]{article}
\usepackage[english]{babel}
\usepackage[utf8]{inputenc}
\usepackage{amsmath,amsfonts,amssymb,amsthm,mathrsfs}
\usepackage{minitoc}

\renewcommand{\leq}{\leqslant}
\renewcommand{\geq}{\geqslant}
\usepackage[mono=false]{libertine}
\useosf
\usepackage{mathpazo}

\usepackage{wasysym}

\usepackage[dvipsnames]{xcolor}
\usepackage[a4paper,vmargin={3.5cm,3.5cm},hmargin={2.3cm,2.3cm}]{geometry}
\usepackage[font=sf, labelfont={sf,bf}, margin=1cm]{caption}
\usepackage{graphicx,graphics}
\usepackage{epsfig}
\usepackage{latexsym}
\usepackage{stmaryrd}
\usepackage{ae,aecompl}

\usepackage[english]{babel}
 \usepackage[colorlinks=true]{hyperref}
\usepackage{pstricks}
\usepackage{enumerate}

\DeclareFixedFont{\beaupetit}{T1}{ftp}{b}{n}{2cm}

\newtheorem{theorem}{Theorem}[]
\newtheorem{definition}{Definition}[]
\newtheorem{prop}[]{Proposition}

\newtheorem{lemma}[]{Lemma}

\theoremstyle{definition}

\newtheorem*{remark}{Remark}

\usepackage{todonotes}

\title{{\bf Surprising identities for the greedy independent set on Cayley trees}} \author{Alice Contat\footnote{Universit\'e Paris-Saclay E-mail: \texttt{alice.contat@universite-paris-saclay.fr} } } 
             
\begin{document}
\date{}            
         \maketitle

\begin{abstract}
We prove a surprising symmetry between the law of the size $G_n$ of the greedy independent set on a uniform Cayley tree $ \mathcal{T}_n$ of size $n$ and that of its complement. We show that $G_n$ has the same law as the number of vertices at even height in $ \mathcal{T}_n$ rooted at a uniform vertex. This enables us to compute the exact law of $G_n$. We also give a Markovian construction of the greedy independent set, which highlights the symmetry of $G_n$ and whose proof  uses a new Markovian exploration of \emph{rooted} Cayley trees which is of independent interest.
\end{abstract}
  
\begin{figure}[!h]
 \begin{center}
 \includegraphics[width=11cm]{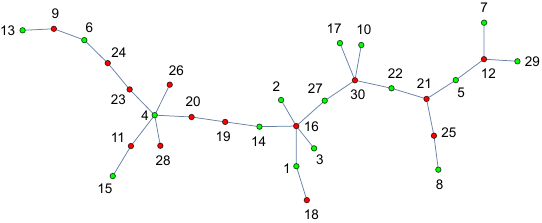}
 \caption{Example of the greedy independent set obtained on a tree of size $30$. The labels represent the order in which vertices are inspected in the construction of the greedy independent set. The green vertices are the active vertices whereas the red vertices are the blocked vertices.\label{fig:exgreedy} }
 \end{center}
 \end{figure}

\section{Introduction}

An \emph{independent} set of a graph $ \mathcal{G} = (V,E)$ is a subset of vertices where no pair of vertices are connected to each other. Finding an independent set of maximal size is a notoriously difficult problem in general \cite{frieze1997algorithmic}. However, using a greedy procedure, we can construct a \emph{maximal} (for the inclusion order) independent set by inspecting the vertices of the graph one by one in a random order, adding the current vertex and blocking its neighbours if it is not connected to any previously added vertex.
More precisely, the vertices are divided in three possible statuses: the undetermined vertices $ \mathcal{U}_k$, the active vertices $ \mathcal{A}_k$ and the blocked vertices $ \mathcal{B}_k$. Initially, we start with $\mathcal{U}_0 = V$ and $ \mathcal{A}_0 = \mathcal{B}_0 = \emptyset$. At step $k \geq 1$, we choose an undetermined vertex $v_k$ uniformly at random, change its status to active and change the status of all its undetermined neighbours to blocked. We stop at  $ \tau = \min \{ k \geq 0, \ \mathcal{U}_k = \emptyset\}$. Note that at each step $k$, no vertices of  $ \mathcal{A}_k$ are neighbours and  $ \mathcal{A}_{\tau}$ is a maximal independent set, which we call the (random) \emph{greedy} independent set, see Figure \ref{fig:exgreedy}. \\
Of course the independent set obtained by the greedy algorithm is usually not \emph{maximum} in the sense that it does not have the maximal possible size. In the case of trees, finding an independent set of maximal size is much simpler than in general. 
However, from a probabilistic or combinatorial point of view the greedy independent set is still worth investigation even on (random) trees.
Greedy independent sets on (random) graphs have been studied extensively with a particular focus on the proportion of vertices of the graph in the greedy independent set called the \emph{greedy independence ratio} or \emph{jamming constant}. Recently, Krivelevich, Mészáros, Michaeli and Shikelman \cite{krivelevich2020greedy} used Aldous' objective method \cite{AS04} to prove under mild assumptions that if a sequence of random finite graphs with a root vertex chosen uniformly at random converges locally, then the sequence of their greedy independence ratios also converges in probability.  \\
Recall that a \emph{Cayley tree} of size $n$ is an unrooted and unordered tree over the $n$ labeled vertices $\{1, \dots , n\}$ and we let $ \mathcal{T}_n$ be a random Cayley tree sampled uniformly at random among the $n^{n-2}$ Cayley trees of size $n$. We shall denote by $ \mathcal{T}_n^{\bullet}$ the rooted tree obtained from $ \mathcal{T}_n$ by distinguishing a vertex uniformly at random.
Using the local limit of $ \mathcal{T}_n^{ \bullet}$ given by Kesten's infinite tree, Krivelevich, Mészáros, Michaeli and Shikelman \cite[Section 6.3]{krivelevich2020greedy} proved the ``intriguing fact" that the asymptotic greedy independence ratio of uniform Cayley trees is $1/2$. Meir and Moon proved in \cite{meir1973expected} that the size of a maximum independent set of a uniform Cayley tree concentrates around $ \rho n$ where $\rho \approx 0.5671$ is the unique solution of $x e^x = 1$.\\
In this note we prove a much stronger, and perhaps surprising statement concerning the size of the greedy independent set on a uniform Cayley tree showing that it has (almost) the same law as that of its complement! We denote by $G_n$ the size of the greedy independent set $ | \mathcal{A}_{\tau}|$ on a uniform Cayley tree $ \mathcal{T}_n$ and  $H_n$ the number of vertices at even height in $ \mathcal{T}_n^{ \bullet}$. Our first observation is that $G_n$ has the same law as $H_n$, which enables us to compute the exact law of $G_n$. 

\begin{theorem}\label{thm:height} The size $G_n$ of the greedy independent set on $ \mathcal{T}_n$ has the same law as the number $H_n$ of vertices at even height in $ \mathcal{T}_n^{ \bullet}.$  
As a consequence, for $1 \leq k \leq n-1$, 
\begin{equation} \label{eq:1} \mathbb{P} (G_n = k) =\mathbb{P}(H_n = k) = \binom{n}{k} \frac{k^{n-k} (n-k)^{k-1}}{n^{n-1}}.\end{equation}
\end{theorem}

The proof of Theorem \ref{thm:height} relies on the invariance of Cayley trees under rerooting at a uniform vertex. The exact computation of the law of $H_n$ is a consequence of a result of \cite{feray2018geometry} on bi-type alternating Galton--Watson trees. This equality in distribution of $G_n$ and $H_n$ suggests that their common law is almost symmetric with respect to $n/2$. But $H_n$ (as well as $G_n$) has a little drift caused by the root vertex of $ \mathcal{T}_n^{ \bullet}$ which is always at even height. Indeed, it follows from Theorem \ref{thm:height} that $G_n/n$ converges in probability to $1/2$ (thus, recovering \cite[Section 6.3]{krivelevich2020greedy}) and we also have a local central limit theorem for $ G_n$: for all $ A \geq 0$, 
\begin{equation}\label{eq:tclloc}  \mathbb{P}  \left( G_n  =  \left\lfloor \frac{n}{2} + x \sqrt{n} \right\rfloor \right) \underset{n \to \infty}{\sim} \frac{1}{\sqrt{n}} \frac{1}{\sqrt{2 \pi (1/4)^2}} \exp \left(- \frac{x^2}{2 \cdot (1/4)^2}\right) \end{equation}
uniformly for $ x \in [-A,A]$. 
We also give a ``Markovian" construction of the greedy independent set which brings to light this symmetry of $G_n$. 
\begin{theorem} \label{thm:sym}
There exists a random variable $ \mathcal{E}_n$ with values in $\{ 0,1\}$ such that we have 
$$ G_n \stackrel{(d)}{=} (n- G_n) + \mathcal{E}_n.$$
Moreover $\mathbb{P} ( \mathcal{E}_n = 1 ) \to 1/4$ as $n$ goes to $\infty$.

\end{theorem}
This symmetry between $G_n$ and $n-G_n$ is striking because the geometry of a greedy independent set and that of its complement are totally different (see Figure \ref{fig:exgreedy}). 
The main idea for the proof of Theorem \ref{thm:sym} is to consider the underlying tree as ``unknown" and to discover it as we build the greedy independent set. See \cite{bermolen2017jamming,brightwell2017greedy,jonckheere2021asymptotic} for similar applications of this technique for other random graph models. In our case, we develop in Section \ref{sec:peeling} a new type of Markovian explorations of uniform Cayley trees which is inspired by Pitman's famous algorithm \cite{pitman1999coalescent} but is more \emph{flexible} in the sense that we can choose which vertex to explore at each step of the procedure. By choosing the next vertex to explore as the next vertex inspected by the greedy algorithm, we connect these Markovian explorations with the construction of the greedy independent set.  We expect these explorations to be applicable to a wider range of contexts, e.g.~we will shed new light on Aldous-Broder algorithm \cite{Ald90,broder1989generating} and Pitman's construction \cite{pitman1999coalescent} of $ \mathcal{T}_n$ using particular cases of our Markovian explorations (see Section \ref{sec:peeling}). Similar explorations are used in a forthcoming work \cite{contat2021parking} on the parking process on Cayley trees. 
\paragraph{Independent sets of edges.} Instead of an independent set of vertices, we could have considered an independent set of edges or a \emph{matching}, that is a set of edges in which no pair of elements have a vertex in common. As for the vertices, we can construct a maximal independent set of edges greedily by inspecting the edges one by one in a uniform random order, keep it if our edge set stays independent and stop once we have inspected all the edges. Denoting by $M_n$ the number of edges kept after applying this algorithm on a uniform Cayley tree $ \mathcal{T}_n$ of size $n$, our Markovian exploration allows us to show that $M_n/n$ concentrates around $3/8$ and to obtain a Central Limit Theorem for $M_n$. We do not include the details here since this  result has already been shown by Dyer, Frieze and Pittel in \cite[Theorem 2]{dyer1993average} by other means.\\
But in the case of plane trees i.e.\ on rooted and ordered unlabeled trees, the invariance under rerooting at a uniform edge plays the same role as the invariance under rerooting at a uniform vertex above and enables us to compute the exact law of the size of the greedy independent set of edges. More precisely, we let $ \widetilde{\mathcal{T}}_n$ be a uniform plane tree i.e.\ rooted and ordered (unlabeled) tree of size $n$ (i.e.\ with $n$ vertices) and let $ \widetilde{M}_n$ be the size of the greedy independent set of edges obtained on $ \widetilde{\mathcal{T}}_n$.
\begin{theorem}\label{thm:plane} The size  $ \widetilde{M}_n$ of the greedy independent set of edges on $ \widetilde{ \mathcal{T}}_n$ has law given by
$$ \mathbb{P} ( \widetilde{M}_n = k) = \dfrac{n}{k} \dfrac{\binom{n+k-2}{2k-1}\binom{n-k-1}{k-1}}{\binom{2n-2}{n-1}} = \frac{n!(n-1)!(n+k-2)!}{k!(2k-1)!(n-2k)!(2n-2)!},$$
for $1 \leq k \leq \lfloor n/2 \rfloor$. As a consequence, for all $ A \geq 0$,
$$ \mathbb{P} \left( \widetilde{M}_n  =  \left\lfloor \frac{n}{3} + x \sqrt{n} \right\rfloor \right) \underset{n \to + \infty}{\sim} \frac{1}{\sqrt{n}} \frac{1}{\sqrt{2 \pi (2/9)^2}} \exp \left(- \frac{x^2}{2 \cdot (2/9)^2}\right) $$
uniformly for $ x \in [-A,A]$. 
\end{theorem}
\noindent We prove this theorem at the end of Section \ref{sec:height}. \\

\noindent Peleg Michaeli has recently informed us that the formula \ref{eq:1} was also independently proved by Alois Panholzer in 2020 in \cite{panholzer2020combinatorial}. 

\noindent  \textbf{Acknowledgments.} The author is very grateful to Mireille Bousquet-Mélou, Vincent Delecroix and Philippe Duchon for enlightening discussions during the 2021 Aléa days. In particular, we are indebted to Mireille Bousquet-Mélou for first computing the law of $G_n$ in Theorem \ref{thm:sym}. We thank Igor Kortchemski for a discussion about bicolored trees. We also thank Matthieu Jonckheere for stimulating discussions and  Nicolas Curien for his constant guidance. I also thank the two anonymous referees for their careful reading and their valuable suggestions that helped clarifying the paper, as well as Peleg Michaeli for pointing \cite{panholzer2020combinatorial} and catching shortcomings in the first version of the paper.

\section{Greedy independent sets and bicolored trees}\label{sec:height}
In this section, we prove Theorems \ref{thm:height} and \ref{thm:plane} by relating the construction of greedy independent sets to bicolored Galton--Watson trees with alternating colors. Our main tool will be invariance of the underlying random tree with respect to rerooting at a uniform vertex or edge. 
\subsection{Independent set of vertices for Cayley trees.}
The following lemma is well known (see for instance \cite[End of Section 1.5]{LG05}) and implies in particular the invariance of a uniform Cayley tree under independent uniform relabeling of the vertices. 
\begin{lemma}\label{lem:invariance}
If $ \mathscr{T}_n$ is a Galton--Watson plane tree with $ \mathrm{Poisson}(1)$ offspring distribution conditioned to have $n$ vertices, then the tree $ \mathcal{T}_n^{\bullet}$ obtained by labeling its vertices by $ \{ 1 , \dots , n\}$ uniformly at random, forgetting the planar ordering but keeping the root vertex, is a uniform rooted Cayley tree.
\end{lemma}
\noindent In particular, if  $ \sigma_n$ is a uniform permutation of $ \{ 1, 2, \dots , n\}$ independent of $ \mathcal{T}_n$, then the tree $\mathcal{T}^{\sigma_n}_n$ obtained by relabeling the vertex $i$ of $ \mathcal{T}_n$ by $ \sigma_n (i)$ for $ 1 \leq i \leq n$ is still a uniform Cayley tree. In the rest of this section, we shall always consider that $ \mathcal{T}_n$ and $ \mathcal{T}_n^{\bullet}$ are built from $ \mathscr{T}_n$ as above.
\begin{remark} By invariance under rerooting, the tree $ \mathcal{T}_n^{ \bullet}$ has the same law, seen as unlabeled rooted tree, as the tree $ \mathcal{T}_n$ rooted at any deterministic vertex $ i \in \{ 1, \dots , n \}$. In the next section, we shall always suppose that our Cayley trees are rooted at the vertex with label $n$, but in this section, it is better to think of them as rooted on a uniform vertex.
\end{remark}
\noindent We decompose the proof of Theorem \ref{thm:height} in two parts. We first prove that $G_n$ has the same law as $ H_n$ and then compute explicitly the law of $H_n$ using bi-type alternating Galton-Watson trees.

\begin{proof}[Proof of the first half of Theorem \ref{thm:height}] The proof of this lemma relies on Lemma \ref{lem:invariance}. We will show that $G_n$ and $H_n$ obey the characteristic same recursive distributional equation. \\
Let us start with $H_n$. 
Recall that $ \mathcal{T}_n^{ \bullet}$ is built from the conditioned Galton–Watson plane tree $ \mathscr{T}_n$ by assigning uniform labels, keeping the root vertex and forgetting about the plane ordering. 
If we denote by $K \geq 0$ the number of vertices of height $2$ in $ \mathscr{T}_n$, by $T_1, \dots , T_K$  the plane trees (ordered from left to right) attached to theses vertices in $ \mathscr{T}_n$, by $N_1, \dots , N_K$ their respective sizes (see Figure \ref{fig:rec}), then by the Markov branching property of the Galton--Watson measure \cite{HM12}, conditionally on $( K, N_1, \dots , N_K)$, the plane rooted trees $(T_1 , \dots , T_K)$ are independent  Galton--Watson plane trees with $ \mathrm{Poisson}(1)$ offspring distribution conditioned to have sizes $(N_1, \dots , N_K)$.
Since the number of vertices at even height in $ \mathscr{T}_n$ is just $1$ (for the root vertex) plus the sum of the number of vertices at even height in every tree $ T_i$ for $ 1 \leq i \leq K$, it follows that
\begin{equation} \label{eq:rdeh} H_{n} \stackrel{(d)}{=} 1 + \sum_{i = 1}^{K} H_{N_i}^{(i)},
\end{equation}
where $(H_{j}^{(i)} : 1 \leq i,j)$ are independent variables also independent of $(K, N_1, \dots , N_K)$  and $H_{j}^{(i)}$ has law $H_j$ for every $ i,j \geq 1$.\\
\begin{figure}[!h]
 \begin{center}
 \includegraphics[width=13cm]{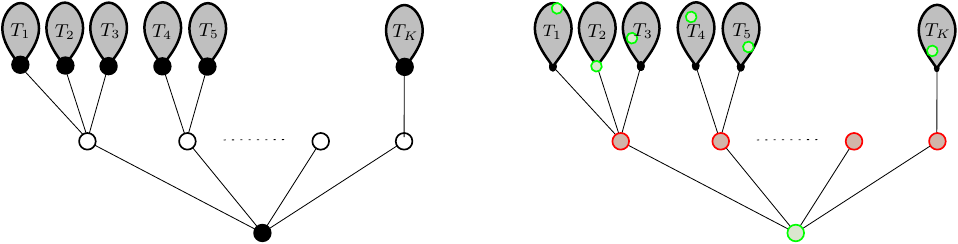}
 \caption{On the left, an illustration of the recursive equation in law for  the number $H_n$ of vertices at even height on $ \mathscr{T}_n$. In black at the bottom, the root vertex, in white its neighbours, and the next vertices to include are the black roots in each tree $T_i$. 
On the right, an illustration of the first step of the greedy algorithm. With our coupling with $ \mathscr{T}_n$, the first vertex that we add in the greedy independent set is the root vertex of $ \mathscr{T}_n$ (in green at the bottom and in red, its neighbours). The next vertices (in green) to inspect in each $T_i$ are then not necessarily the root vertices of $T_i$.
   \label{fig:rec} } 
 \end{center}
 \end{figure} \\ 
Let us now move on to the size of the greedy independent set. By construction, it is built first by including a uniform vertex $V$ with label in $ \{ 1, \dots , n\}$ and  blocking  its neighbours (which are the vertices at distance $1$ from the vertex $V$). We can assume that this vertex is actually the root vertex of $ \mathcal{T}_n^{\bullet}$ (or equivalently of $ \mathscr{T}_n$) since it is a uniform vertex of $ \{ 1, \dots , n\}.$ Using the same notation as above (see Figure \ref{fig:rec}), the crucial observation is that the greedy independent set is obtained by joining the existing root vertex of $ \mathscr{T}_n$ together with the independent sets obtained by applying the greedy algorithm on the trees $ T_1, \dots , T_K$ independently. The difference with the case of $H_n$ above is that, in each tree $T_i$, the next vertex to inspect is not necessarily the root of $ T_i$ (induced by $ \mathscr{T}_n$)  but a ``new" uniform vertex of $ T_i$ (see Figure \ref{fig:rec}). But by invariance of uniform Cayley trees under uniform rerooting, we still have
\begin{equation} \label{eq:rdeg}  G_{n} \stackrel{(d)}{=} 1 + \sum_{i = 1}^K G_{N_i}^{(i)}
\end{equation}  
where $(G_{j}^{(i)} : 1 \leq i,j)$ are independent variables also independent of $(K, N_1, \dots , N_K)$ and $G_{j}^{(i)}$ has law $G_j$ for $ i,j \geq 1$.\\
Moreover, the equations \eqref{eq:rdeh} and \eqref{eq:rdeg} characterize the law of $H_n$ and $G_n$ since $ N_1 + \dots + N_K \leq n-1$ almost surely. Hence $G_n$ and $H_n$ have the same law and we get the desired result. 
\end{proof}

\begin{proof}[Proof of the second half of Theorem \ref{thm:height}]By Lemma \ref{lem:invariance}, the variable $H_n$ has the same law as the number of vertices at even height in a Galton--Watson tree with $ \mathrm{Poisson}(1)$ offspring distribution conditioned to have $n$ vertices. To  compute it, we artificially introduce two types of vertices (white vertices and black vertices) and let $ \mathscr{T}_{b}$ be a two-type alternating Galton--Watson tree with a black root and with $\mathrm{Poisson}(1)$ offspring distribution (for both types of vertices), so that all vertices at even height are black and have white children, and all vertices at odd height are white and have black children. We denote by $N_b$ (resp.\ $N_w$)  the number of black (resp.\ white) vertices in $ \mathscr{T}_b$.
Using the result of \cite{chaumont2016coding} and more precisely \cite[Corollary 3.4]{feray2018geometry} and denoting by $S_j$ the sum of $j$  i.i.d.\ random variables with law $ \mathrm{Poisson}(1)$ for $j \geq 1$, so that $S_j$ has law $ \mathrm{Poisson}(j)$, we obtain, for all $ k \geq 1$, 
\begin{eqnarray*}
\mathbb{P}( H_n = k) &=& \mathbb{P}( N_b = k,  N_w =  n-k   | N_b + N_w = n) \\
&=& \frac{n}{k} \frac{ \mathbb{P}(S_k = n-k) \mathbb{P}(S_{n-k} = k-1)}{ \mathbb{P}(S_n = n-1)} \\
&=& 	\binom{n}{k} \frac{k^{n-k}(n-k)^{k-1}}{n^{n-1}}.
\end{eqnarray*}\end{proof}
A straightforward application of Stirling's formula gives Equation \ref{eq:tclloc}. We then easily deduce the local central limit theorem and the law of large numbers.

\subsection{Independent set of edges for plane trees.} As mentioned in the introduction, this construction of the greedy independent set of vertices on uniform Cayley trees, which are invariant under rerooting at a uniform vertex, suggests a similar construction for the greedy independent set of edges on uniform plane trees which are invariant under rerooting at a uniform edge. 
\noindent Recall that we denote by $ \widetilde{\mathcal{T}}_n$ a uniform plane tree of size $n$ (i.e.\ with $n$ vertices).  It can be seen as a graph which is properly embedded in the plane, has only one face of degree $2n-2$, and is rooted at the oriented edge going from the root vertex to its leftmost child (see Figure \ref{fig:planeroot}). The crucial observation is that if, conditionally on $ \widetilde{ \mathcal{T}}_n$, we let $ \vec{e}$ be a uniform oriented edge of $ \widetilde{ \mathcal{T}}_n$, then the tree $ \widetilde{ \mathcal{T}}_n^{ \vec{e}}$ obtained be rerooting the tree $ \widetilde{ \mathcal{T}}_n$ at the edge $ \vec{e}$ is still a uniform plane tree.

\begin{figure}[!h]
 \begin{center}
 \includegraphics[height=4.5cm]{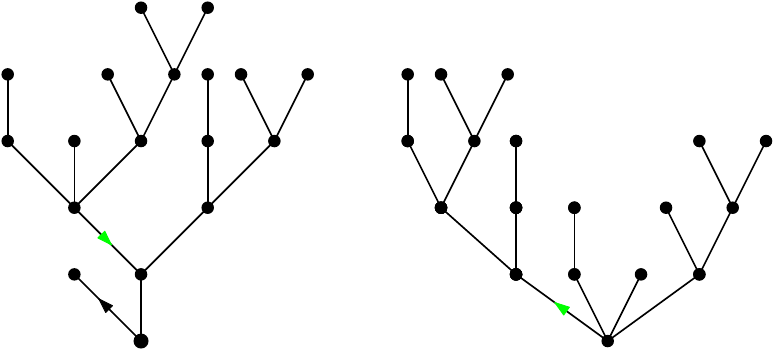}
 \caption{On the left, a plane tree rooted at the black oriented edge. On the right, the usual representation of this plane tree rerooted at the green oriented edge. \label{fig:planeroot}}
 \end{center}
 \end{figure}

Moreover, as uniform Cayley trees, uniform plane trees of size $n$ can be seen as conditioned Galton--Watson trees with the appropriate offspring distribution (see for instance \cite{Ald93}).
\begin{lemma}\label{lem:invPlaneTreee}
A Galton--Watson plane tree with $\mathrm{Geom}(1/2)$ offspring distribution conditioned to have $n$ vertices is a uniform (rooted) plane tree of size $n$, where a random variable $X$ has law $\mathrm{Geom}(1/2)$ if for all $ k \geq 0$, we have $\mathbb{P} (X = k) = 2^{-k-1}. $

\end{lemma}

Using this lemma and the invariance under  rerooting at a uniform oriented edge, we now prove Theorem \ref{thm:plane}.

\begin{proof}[Proof of Theorem \ref{thm:plane}] 
By invariance under rerooting at a uniform oriented edge, we can assume that our greedy algorithm on $ \widetilde{\mathcal{T}}_n$ first includes the root edge, blocks its ``neighbouring edges" i.e.\ the edges adjacent to one of its endpoints (see Figure \ref{fig:plane}, left). We denote by $K_2$ the number of children of the root vertex minus $1$ (or number of brothers of the root edge), by $K_1$ the number of children of the first child of the root, and by $T_1^1, \dots ,T_{K_1}^1$ and $T_1^2, \dots T_{K_2}^2$ the trees attached to theses vertices (ordered from left to right, see Figure \ref{fig:plane}), which have respective sizes $N_1^1, \dots ,N_{K_1}^1$ and $N_1^2, \dots, N_{K_2}^2$. Then, conditionally on $(K_1, K_2, N_1^1, \dots ,N_{K_1}^1, N_1^2, \dots,N_{K_2}^2)$, the trees $(T_1^1, \dots T_{K_1}^1, T_1^2, \dots ,T_{K_2}^2)$ are independent Galton--Watson trees with $ \mathrm{Geom}(1/2)$ offspring distribution conditioned to have respective sizes $(N_1^1, \dots ,N_{K_1}^1, N_1^2, \dots, N_{K_2}^2)$. By the same argument as in the proof of Theorem \ref{thm:height}, we deduce that $ \widetilde{M}_n$ satisfies the following recursive distributional equation 
\begin{equation} \label{eq:rdeplane} \widetilde{M}_{n} \stackrel{(d)}{=} 1 + \sum_{i = 1}^{K_1} \widetilde{M}_{N_i^1}^{(i)}+ \sum_{i = 1}^{K_2} \widetilde{M}_{N_i^2}^{(K_1+i)},
\end{equation}
where $( \widetilde{M}_{j}^{(i)} : 1 \leq i,j)$ are independent variables also independent of $$(K_1, K_2, N_1^1, \dots ,N_{K_1}^1, N_1^2, \dots,N_{K_2}^2)$$ and $ \widetilde{M}_{j}^{(i)}$ has law $ \widetilde{M}_j$ for $1 \leq i,j$. \\
As in the proof of Theorem \ref{thm:height} we interpret the last recursive distributional equation using two-type Galton--Watson trees. Specifically, let $ \mathbb{T}$ be a  two-type alternating Galton--Watson tree, with green and red vertices, with a green root and where the red vertices have Bernoulli of parameter $1/2$ offspring distribution and the offspring distribution of the green vertices is the law of the sum of two independent variables with law $ \mathrm{Geom} (1/2)$. We denote by  $N^g$ (resp.\ $N^r$)  the number of green (resp.\ red) vertices in $ \mathbb{T}$.
Writing $ N^g_{n}$ for the number of green vertices in the tree $ \mathbb{T}$ conditioned to have $n-1$ vertices in total, then $ N^g_{n}$ obeys the same recursive equation \eqref{eq:rdeplane} as $ \widetilde{M}_n$ (see Figure \ref{fig:plane}).

\begin{figure}[!h]
 \begin{center}
 \includegraphics[width=13cm]{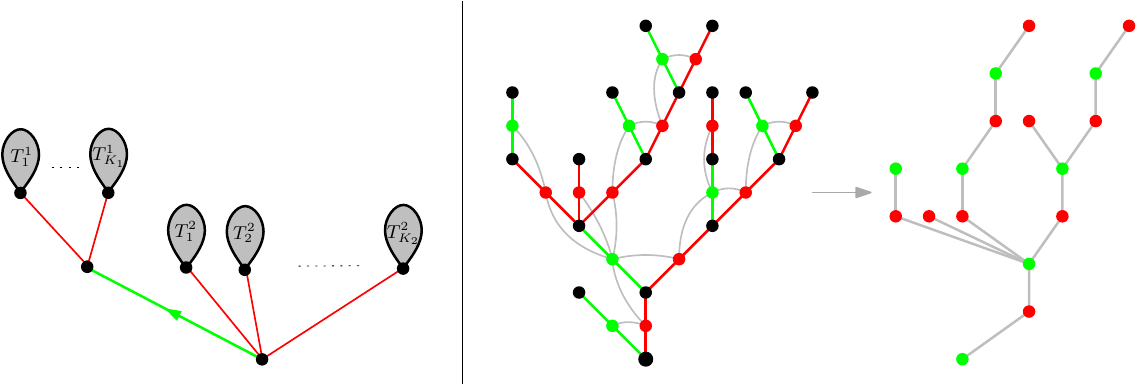}
 \caption{On the left, an illustration of the first step of the greedy algorithm for the independent set of edges. The first edge that we add in the independent set is the green root edge, and we block its neighbouring edges (red). The next edge to inspect is a uniform edge in each tree $T_i^{j}$, but can be taken as the root edge by invariance under rerooting.
 On the right, a plane tree $ \mathbf{t}$ with black vertices and with edges colored as follows: we color the root edge in green, its neighbouring edges in red and reapply the procedure in each tree by taking first the root edge (in green). Its corresponding bi-type alternating plane tree $ \mathbf{t}_g$ is obtained by considering the edges of $ \mathbf{t}$ as the vertices of  $ \mathbf{t}_g$: the children of a green vertex in $ \mathbf{t}_g$ correspond to its children followed by its brothers in $ \mathbf{t}$, and a red vertex has a green child in $ \mathbf{t}_g$ if the red corresponding edge has (at least) a child in $ \mathbf{t}$. 
   \label{fig:plane}}
 \end{center}
 \end{figure} 
\noindent Using again \cite[Corollary 3.4]{feray2018geometry} and denoting by $ \widetilde{S}_j^g$ the sum of $j$  i.i.d.\ random variables with law $ \mathrm{Geom}(1/2)$ and $ \widetilde{S}_j^{r}$ the sum of $j$ i.i.d.\ random variables with Bernoulli law of parameter $1/2$ for $j \geq 1$, so that $ \widetilde{S}_j^{r}$ has binomial distribution of parameters  $ j$ and  $1/2$, we obtain, for all $ k \geq 1$, 
\begin{eqnarray*}
\mathbb{P}( \widetilde{M}_n = k) &=& \mathbb{P}( N^g = k,  N^r =  n-1-k   | N^g + N^r = n-1) \\
&=& \frac{1}{k} \frac{ \mathbb{P}( \widetilde{S}_{2k}^{g} = n-1-k) \mathbb{P}( \widetilde{S}_{n-1-k}^{r} = k-1)}{ \mathbb{P}( \widetilde{S}_{n}^{g} = n-1) } \\
&=& \dfrac{n}{k} \dfrac{\binom{n+k-2}{2k-1}\binom{n-k-1}{k-1}}{\binom{2n-2}{n-1}}.
\end{eqnarray*}
\end{proof}

 \section{Markovian explorations of a rooted tree} \label{sec:peeling}
In this section, we introduce the Markovian explorations of Cayley trees which we will use to prove the symmetry of the law of $G_n$ (Theorem \ref{thm:sym}) and which we shall call ``peeling explorations" by analogy with the peeling process in the theory of random planar maps, see \cite{CurStFlour}. 
Given a Cayley tree $ \mathbf{t}$, an \emph{exploration} of $ \mathbf{t}$ will be a sequence of forests $( \mathbf{f}_0 , \dots, \mathbf{f}_{n-1})$ starting from the forest $ \mathbf{f}_0$ made of $n$ isolated vertices and ending at $ \mathbf{f}_{n-1} = \mathbf{t}$, such that we pass from $ \mathbf{f}_i$ to $ \mathbf{f}_{i+1}$ by adding one edge. In our setup, the edge to add at each step will be the edge linking the vertex \emph{peeled} at step $i$ to its parent. To do so, we thus need to root our tree $ \mathbf{t}$ by distinguishing a vertex and orienting all edges towards it. That is why in the rest of the paper (and contrary to the previous section) we will always see a Cayley tree $ \mathbf{t}$ as rooted at the vertex of index $n$. \\
Specifically,  for $0 \leq k \leq n-1$, we let $\mathcal{F}^{*}(n,k)$ be the set of (unordered) rooted forests on labeled vertices $\{ 1, \dots, n \} $ with exactly $k$ edges where one of the trees is distinguished: in the following, the distinguished tree and its vertices are seen as being \emph{blue} whereas the other trees and vertices are \emph{white}, see Figure \ref{fig:peeling}. 
We denote by  $\mathcal{F}_n^{*} = \cup_{k=0}^{n-1} \mathcal{F}^{*}(n,k)$ the set of such rooted forests on $\{ 1, \dots, n \}$. Below, the blue component will correspond to the component of the ``real" root of the tree in our peeling exploration of rooted Cayley tree. Given a forest $ \mathbf{f} \in \mathcal{F}^{*}(n,k)$, we say that a rooted Cayley tree $ \mathbf{t}$ \emph{contains} $ \mathbf{f}$ if $ \mathbf{t}$ can be obtained from $ \mathbf{f}$ by adding edges between each white root of a tree and another compatible vertex i.e.\ a vertex contained in a different tree of $ \mathbf{f}$.  Moreover, if $v_1$ is the root of a white tree and $v_2$ is a vertex of another tree, we denote by $\mathbf{f}_{v_1 \to v_2}$ the forest  obtained from $ \mathbf{f}$ by adding an edge from $v_1$ to $v_2$ and coloring the resulting component with the color of $v_2$. Assume now that we have a function $\mathfrak{a}$ called the \emph{peeling algorithm} which associates $ \mathbf{f} \in \mathcal{F}^{*}_n \setminus \mathcal{F}^* (n, n-1)$ with a white root of a tree of the forest $ \mathbf{f}$.\begin{figure}[!h]
 \begin{center}
 \includegraphics[width = 13cm]{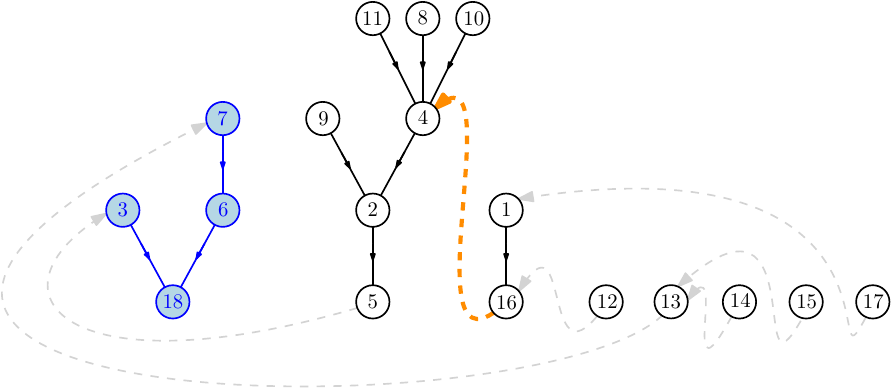}
 \caption{Example of a forest $ \mathbf{F}_{10}^{\mathfrak{a}}$ which can be obtained after $10$ steps of exploration of a given tree $ \mathbf{t}$ of size $18$. The edges which are still unknown are in dashed. If the vertex labeled $ \mathfrak{a}( \mathbf{F}_{10}^{\mathfrak{a}}) = 16$ is the next peeled vertex, then the next edge to be added will be the dashed orange edge. \label{fig:peeling}}
 \end{center}
 \end{figure}
\begin{definition}
Let $\mathbf{t}$ be a Cayley tree with $n$ vertices rooted at the vertex labeled $n$. The peeling exploration of $ \mathbf{t}$ with algorithm $\mathfrak{a}$ is the sequence
\begin{eqnarray*}
\mathbf{F}_{0}^{\mathfrak{a}} \longrightarrow \mathbf{F}_{1}^{\mathfrak{a}} \longrightarrow \ldots \longrightarrow \mathbf{F}_{n-2}^{\mathfrak{a}} \longrightarrow \mathbf{F}_{n-1}^{\mathfrak{a}}
\end{eqnarray*}
obtained by starting from $\mathbf{F}_0^{\mathfrak{a}}$ the unique element of $\mathcal{F}^{*} (n,0) $ with isolated vertices, which are all white except the blue vertex $n$, and at step $i \geq 1$, the forest $\mathbf{F}_{i}^{\mathfrak{a}}$ is obtained from $\mathbf{F}_{i-1}^{\mathfrak{a}}$ and  $\mathfrak{a} (\mathbf{F}_{i-1}^{\mathfrak{a}})$ as follows: if $v$ is the parent of $\mathfrak{a} (\mathbf{F}_{i-1}^{\mathfrak{a}})$ in $ \mathbf{t}$, then we add the edge from $\mathfrak{a} (\mathbf{F}_{i-1}^{\mathfrak{a}})$ to $v$ to $ \mathbf{F}_{i-1}^{ \mathfrak{a}}$ and color the vertices of the resulting component with the color of $v$ to obtain $ \mathbf{F}_{i}^{ \mathfrak{a}} = \left(\mathbf{F}_{i-1}^{ \mathfrak{a}} \right)_{\mathfrak{a} (\mathbf{F}_{i-1}^{\mathfrak{a}}) \to v}.$
\end{definition}
Notice that $\mathbf{F}_{n-1}^{\mathfrak{a}}$ is the tree $ \mathbf{t}$ with blue vertices, regardless of the choice of $\mathfrak{a}$. 
Moreover, when $\mathbf{t}=\mathcal{T}_n$ is a uniform Cayley tree rooted at $n$, then the sequence $(\mathbf{F}_{i}^{\mathfrak{a}})_{0 \leq i \leq n-1}$ is a Markov chain:  
\begin{prop} \label{prop:markovtree} Fix a peeling algorithm $ \mathfrak{a}$. If $\mathcal{T}_n$ is a uniform Cayley tree of size $n$ rooted at $n$, then the exploration $(\mathbf{F}_{i}^{\mathfrak{a}})_{0 \leq i \leq n-1}$ of $ \mathcal{T}_n$ with algorithm $ \mathfrak{a}$ is a Markov chain  whose probability transitions are described as follows: for $i \geq 0$, conditionally on $\mathbf{F}_{i}^{\mathfrak{a}}$ and on $\mathfrak{a} (\mathbf{F}_{i}^{\mathfrak{a}})$, we denote by $m \geq 1$ the size of the connected component of $\mathfrak{a} (\mathbf{F}_{i}^{\mathfrak{a}})$ and by $\ell \geq 1$ the number of blue vertices in $ \mathbf{F}_{i}^{ \mathfrak{a}}$. \begin{itemize}
 \item For every blue vertex $v$, the parent of $ \mathfrak{a} (\mathbf{F}_{i}^{\mathfrak{a}})$ is $v$ with probability $(\ell+m)/(\ell n)$ and in that case, we add an edge from $\mathfrak{a} (\mathbf{F}_{i}^{\mathfrak{a}})$ to $v$  and color the resulting component  blue  to obtain $	 \mathbf{F}_{i+1}^{ \mathfrak{a}}$.
 \item For every white vertex $v$ which does not belong to the component of $\mathfrak{a} (\mathbf{F}_{i}^{\mathfrak{a}})$, the parent of $ \mathfrak{a} (\mathbf{F}_{i}^{\mathfrak{a}})$ is $v$ with probability $1/n$ and in that case we add an edge from $\mathfrak{a} (\mathbf{F}_{i}^{\mathfrak{a}})$ to $v$  to obtain $ \mathbf{F}_{i+1}^{ \mathfrak{a}}$.
\end{itemize}
\end{prop}
This proposition is a direct corollary of the following lemma.
\begin{lemma} \label{lem:nbCayley} If $ \mathbf{f} \in \mathcal{F}^{*} (n,k)$ is a forest on $\{ 1, \dots, n\}$ with $k$ edges and $\ell \geq 1$ blue vertices (hence $n-k-1$ white trees and a blue tree), then the number $N( \mathbf{f})$ of rooted Cayley trees containing $ \mathbf{f}$ is equal to $\ell n^{n-k-2}$. 
\end{lemma}

\begin{proof} The proof is similar to that of Pitman in \cite[Lemma 1]{pitman1999coalescent}. We introduce the number $N^*(\mathbf{f} )$ of \emph{refining} sequences of $ \mathbf{f}$, that is the number of sequences $(\mathbf{f}_j , \dots , \mathbf{f}_{1})$ where $ \mathbf{f}_j = \mathbf{f}$ and for $i \leq j$, the forest $ \mathbf{f}_i$ has a blue tree and $i-1$ white trees, and $ \mathbf{f}_{i-1}$ can be obtained from $ \mathbf{f}_i$ by adding an edge as above. Note that $j = n-k$ since  $\mathbf{f}$ has $n-k-1$ white trees. Given a fixed target $ \mathbf{f}_1$, any forest $ \mathbf{f}_j$ that contains a $ \mathbf{f}_1$ has $j-1$ fewer edges  than $ \mathbf{f}_1$. Hence there are $(j-1)!$ refining sequences from $ \mathbf{f}_j$ to $ \mathbf{f}_1$  and $N^*( \mathbf{f}_j) = (j-1)! N( \mathbf{f}_j)$. 
Now we prove the result on $N( \mathbf{f})$ by induction over $j-1$, the number of white trees of $ \mathbf{f}$. When $j =1$, then $ \mathbf{f} \in \mathcal{F}^* (n,n-1)$ is simply a blue rooted tree and the result is clear.

Suppose the result holds for some $ j-1 \geq 0$. Let $ \mathbf{f} \in \mathcal{F}^* (n,n-j-1)$ be a forest on $\{ 1, \dots, n\}$ with $n_1 \geq 1$ blue vertices and $j$ white trees $( \mathbf{t}_2, \dots, \mathbf{t}_{j+1})$ of size $(n_2, \dots , n_{j+1})$. Then, if $r_i$ is the root of the white tree $ \mathbf{t}_i$,
\begin{itemize}
\item for each blue vertex $v_1$, by the induction hypothesis, there are $ N^*( \mathbf{f}_{r_i \to v_1}) = (n_1 + n_i) n^{j-2} (j-2)!$  refinements of $ \mathbf{f}_{r_i \to v_1}$ and there are $n_1$ such vertices $v_1$,
\item and for each vertex $v_j$ in a white tree $ \mathbf{t}_j$ with $ i \neq j$,  by the induction hypothesis, there are $ N^*( \mathbf{f}_{r_i \to v_j}) = n_1 n^{j-2} (j-2)!$ refinements of $ \mathbf{f}_{r_i \to v_j}$ and there are $n -(n_i + n_1)$ such vertices $v_j$.
\end{itemize}
In total, the number of refinements of $ \mathbf{f}$ is 
$$ N^*( \mathbf{f}) = \sum_{i = 2}^{j+1} n_1 (n_1 + n_i) n^{j-2} (j-2)! + (n -(n_i + n_1))n_1 n^{j-2} (j-2)! = (j-1)! n_1 n^{j-1}, $$
and we get the desired result.
\end{proof}
Now we can prove Proposition \ref{prop:markovtree}.
\begin{proof}[Proof of Proposition \ref{prop:markovtree}] It suffices to notice that since $ \mathcal{T}_n$ is uniform, for all $i \geq 0$, conditionally on $ \mathbf{F}_i^{\mathfrak{a}}$, the tree $ \mathcal{T}_n$ is a uniform tree among those which contain $ \mathbf{F}_i^{\mathfrak{a}}$. Hence, for every (compatible) vertex $v$, 
\begin{equation*}
\mathbb{P} ( \textbf{F}_{i+1}^{\mathfrak{a}} = ( \textbf{F}_{i}^{\mathfrak{a}})_{\mathfrak{a} ( \textbf{F}_{i}^{\mathfrak{a}}) \to v} | \textbf{F}_{i}^{\mathfrak{a}}, \mathfrak{a} ( \textbf{F}_{i}^{\mathfrak{a}})) = \frac{|\{ \mathbf{t} \mbox{ s.t } \mathbf{t} \text{ contains } ( \textbf{F}_{i}^{\mathfrak{a}})_{\mathfrak{a} ( \textbf{F}_{i}^{\mathfrak{a}}) \to v} \}|}{| \{ \mathbf{t} \mbox{ s.t } \mathbf{t} \text{ contains } \textbf{F}_{i}^{\mathfrak{a}} \} |}.
\end{equation*}
Using Lemma \ref{lem:nbCayley}, we recognize the transition probabilities given in Proposition \ref{prop:markovtree} and obtain the desired result.
\end{proof}
The strength of Proposition \ref{prop:markovtree} is that different peeling algorithms yield different explorations (hence different types of information) of the same underlying tree. Let us illustrate this with several examples of peeling algorithms. 

\paragraph{Pitman's algorithm and $ \mathfrak{a}_{ \mathrm{unif}}$.} A natural choice of algorithm $ \mathfrak{a}$ is, given a forest $ \mathbf{f}$, to choose a root of a white tree of $ \mathbf{f}$ uniformly at random for $ \mathfrak{a}_{ \mathrm{unif}} ( \mathbf{f})$. 
This does not seem to enter our setup since $  \mathfrak{a}_{ \mathrm{unif}}$ is not a deterministic function of $ \mathbf{f}$. However we can imagine that we first condition on the randomness involved in $  \mathfrak{a}_{ \mathrm{unif}}$ making it deterministic. Once $ \mathfrak{a}_{ \mathrm{unif}}$ is fixed, we apply it to a random Cayley tree $ \mathcal{T}_n$, thus independent of the choice of $ \mathfrak{a}_{ \mathrm{unif}}$. The Markov chain obtained in Proposition~\ref{prop:markovtree} to this peeling is reminiscent of (but not identical to) Pitman's construction \cite{pitman1999coalescent} of uniform rooted Cayley trees, which we quickly recall: Start from the forest made of the  $n$ isolated vertices $ \{ 1, \dots , n\}$ and at step $ 1 \leq k \leq n-1$, pick a vertex $V_k$ uniformly at random and then pick a root $R_k$ uniformly at random among the $n-k$ trees which do not contain $V_k$, and add the directed edge from $R_k$ to $V_k$. Note that in Pitman's construction, we first pick a uniform vertex then a uniform compatible root and the ``real" root of the Cayley tree is found out only at the end of the exploration whereas in our peeling exploration with algorithm $ \mathfrak{a}_{ \mathrm{unif}}$, we first fix the root as the vertex labeled $n$ and during the exploration we choose uniformly a root of a white tree whose parent is chosen almost uniformly among the compatible vertices.
\paragraph{Building branches and Aldous-Broder algorithm.}
Let us mention another choice of algorithm which sheds new light on Aldous-Broder's construction \cite{Ald90,broder1989generating} of $ \mathcal{T}_n$ by using a random walk on the complete graph \emph{with loops}. We assume in this paragraph that $ n \geq 2$. Given a forest $ \mathbf{f} \in \mathcal{F}^{*}_n$, let $ \mathfrak{a}_{AB} ( \mathbf{f})$ be the root of the tree which contains the white vertex with the smallest label. For example, starting at $\mathbf{F}_0^{  \mathfrak{a}_{AB}}$, then $ \mathfrak{a}_{AB}(\mathbf{F}_i^{ \mathfrak{a}_{AB}})$ is the root of the component of $1$ until the exploration hits the blue root $n$. Let $T_1^{\mathfrak{a}_{AB}} = \min \{ i \geq 1 \mbox{ s.t. } 1 \mbox{ is a blue vertex in } F_{i-1}^{ \mathfrak{a}_{AB}}\}$ be the length of the first ``branch" built in this exploration. Then, for $k \geq 2$,  
$$ \mathbb{P} (T_1^{\mathfrak{a}_{AB}} = k) =\frac{k}{n} \prod_{i = 2}^{k-1} \left( 1 - \frac{i}{n}\right).$$
Recall now Aldous-Broder's algorithm \cite{Ald90,broder1989generating} on the complete graph on $n$ vertices, that is the graph with vertex set $\{ 1, \dots , n\}$, and all possible edges including the loops $(i,i)$: Consider a random walk $(X_i : i \geq 0)$ starting at $X_0 :=n$ and where the $X_i$'s are i.i.d.\ uniform random variables on $\{ 1, \dots n \}$ for $i \geq 1$. 
We denote by $ \tau_k := \inf \{i : X_i = k\} < \infty$ for any $ k \in \{ 1, \dots , n\}$.
The walk induces the tree $ \mathcal{T}^{AB}$ which consists of the tree on $\{ 1, \dots , n\}$ with  the edges $(X_{\tau_k - 1}, X_{\tau_k})$ for $ 1 \leq k \leq n-1$ (we only take the edges which lead to newly discovered vertices). Consider $T_1^{AB} = \min \{ i \geq 1 \mbox{ s.t. } X_{i} \in \{ X_0 , \dots , X_{i-1} \} \}$ the time of the first repetition of the walk $(X_i : i \geq 0).$ Then by construction,  $\mathbb{P} ( T_1^{AB} = 1) = 1/n$ and for $k \geq 2$, 
\begin{eqnarray*}
\mathbb{P} ( T_1^{AB} = k) = \frac{k}{n} \prod_{i = 1}^{k-1} \left( 1 - \frac{i}{n}\right) = \frac{n-1}{n} \mathbb{P} (T_1^{\mathfrak{a}_{AB}} = k) = \mathbb{P}( T_1^{AB} \neq 1) \cdot \mathbb{P} (T_1^{\mathfrak{a}_{AB}} = k).
\end{eqnarray*}
We deduce that conditionally on $X_1 \neq n$, the law of the length  $T_1^{AB}$ of the first branch is identical to that of the length of the branch linking $n$ and $1$ in $ \mathcal{T}_n$, see \cite{Ald91a}. A similar result holds for the next branches in both constructions. This has been already observed by Camarri and Pitman in the more general context of $p$-trees in \cite[Corollary 3]{CaPi2000}.

\paragraph{Greedy construction.}In the next section, we use a  peeling algorithm (with additional decorations) tailored to the construction of the greedy independent set on $ \mathcal{T}_n$.

\section{Markovian construction of the greedy independent set} \label{sec:sym}

This section is devoted to the proof of Theorem \ref{thm:sym} which use our Markovian exploration of Cayley trees.
\subsection{Markovian construction and its transitions}
Recall from the introduction the greedy algorithm: given a Cayley tree, we inspect its vertices sequentially in a uniform random order and at each step, if the considered vertex is undetermined, we change its status to active and change the status of all its undetermined neighbours to blocked. 
By Lemma \ref{lem:invariance}, we have an invariance property of $ \mathcal{T}_n$ under independent uniform relabeling. Hence we can and will directly use the labels of $ \mathcal{T}_n$  to define the order of exploration of the greedy independent set on $ \mathcal{T}_n$.  \\
Recall that we rooted our tree at the vertex $n$. The idea here is to link the greedy construction to a peeling exploration of $ \mathcal{T}_n$: we explore and inspect at each step a vertex and update its status and possibly \emph{that of its parent in the tree but not these of its children}. In particular, the root $n$ of the tree is possibly the last vertex to be considered (when it is not previously blocked). \\
More precisely, given the Cayley tree $ \mathcal{T}_n$ rooted at $n$, we divide its vertices into three statuses (undetermined, active or blocked) as before and set 
$$ \widetilde{\mathcal{U}}_0^n = \{ 1, \dots , n \} , \qquad  \widetilde{\mathcal{A}}_0^n = \widetilde{\mathcal{B}}_0^n = \emptyset.$$ 
Then inductively at step $i$, we inspect $v_i$ the \emph{undetermined} vertex with the smallest label  and we let $w_i$ be its parent in $ \mathcal{T}_n$ (unless $v_i = n$). Then we update their statuses as follows:
\begin{itemize}
\item If $v_i = n$, then the vertex $n$ becomes active i.e.\ $ \widetilde{\mathcal{U}}_{i+1}^n = \emptyset$ and $ \widetilde{\mathcal{A}}_{i+1}^n = \widetilde{\mathcal{A}}_i^n \cup \{v_i\}$
\item If $w_i$ is undetermined, then $v_i$ becomes active (it is taken in the greedy independent set) and $w_i$ becomes blocked i.e.\ $ \widetilde{\mathcal{U}}_{i+1}^n = \widetilde{\mathcal{U}}_{i}^n \setminus\{ v_i ,w_i\}$ and $ \widetilde{\mathcal{A}}_{i+1}^n = \widetilde{\mathcal{A}}_{i}^n \cup \{ v_i \}$ and  $ \widetilde{\mathcal{B}}_{i+1}^n = \widetilde{\mathcal{B}}_{i}^n \cup \{ w_i \}$.
\item If $w_i$ is blocked (resp.\ active), then $v_i$ becomes active (resp.\ blocked), that is $ \widetilde{\mathcal{U}}_{i+1}^n = \widetilde{\mathcal{U}}_{i}^n \setminus\{ v_i \}$ and $ \widetilde{\mathcal{A}}_{i+1}^n = \widetilde{\mathcal{A}}_{i}^n \cup \{ v_i \}$ (resp.\ $ \widetilde{\mathcal{B}}_{i+1}^n = \widetilde{\mathcal{B}}_{i}^n \cup \{ v_i \}$).
\end{itemize} 
 \begin{figure}[!h]
 \begin{center}
 \includegraphics[width=13cm]{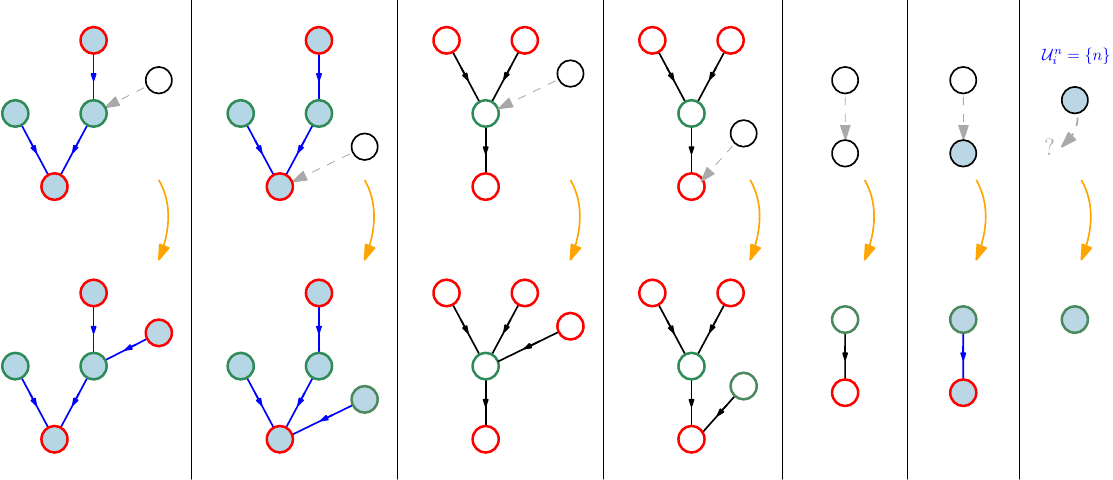}
 \caption{Illustration of the possible transitions in the peeling exploration with $ \mathfrak{a}_{ \mathrm{Greedy}}$ together with the updating of the status of the vertices. The interior color represents the color of the vertices (blue or white) in the peeling exploration and the boundary color represents the status of the vertices in the greedy construction of the independent set: black for the undetermined vertices, green for the active vertices and red for the blocked vertices. The new edge which we explore (between $v_i$ and $w_i$) is in dotted gray.   When $v_i = n$, there is no peeling step but $v_i$ is the last undetermined vertex and it becomes active. 
 \label{fig:transitions}}
 \end{center}
 \end{figure} 
As before, at each step $i$, the set $ \widetilde{\mathcal{A}}_i^n$ is an independent set and we stop the process when all vertices are active or blocked i.e.\ at time $ \theta_n = \inf \{ i \geq 0, \ \widetilde{\mathcal{U}}_i^n = \emptyset\}$. It is then easy to check that the random subset $ \widetilde{ \mathcal{A}}_{ \theta_n}^n$ is equal to the greedy independent set $ \mathcal{A}_{\tau}$ defined in the introduction (if we inspect vertices according to their labels in the tree). 
The advantage of the above construction over the one of the introduction where we update the status of all neighbours of the inspected vertex is that it can be seen as a Markovian exploration of $ \mathcal{T}_n$ in the sense of the preceding section (where we additionally keep track of the status of the vertices). Indeed, it is equivalent to the peeling of $ \mathcal{T}_n$ started at the forest $ \mathbf{F}_{0}^{ \mathfrak{a}_{ \mathrm{Greedy}}}$ made of an undetermined blue vertex with label $n$ and $n -1$ isolated white undetermined vertices and where at step $i+1$, we peel the vertex $$ v_{i} = \min\  \widetilde{\mathcal{U}}_i^n := \mathfrak{a}_{\mathrm{Greedy}}( \mathbf{F}_i^{ \mathfrak{a}_{\mathrm{Greedy}}})$$ 
the undetermined vertex with the smallest label and update the status of the vertices as above, see Figure \ref{fig:transitions}.
Recall the two possible colors for the vertices during a peeling exploration: the vertices which are in the connected component of the vertex labeled $n$ which is the ``real root" of the tree are \emph{blue} whereas the other vertices are \emph{white}. We make the following remarks which are easily checked by induction: with the possible exception of the vertex $n$, the undetermined vertices are white whereas the active or blocked vertices can be white or blue. Moreover at each step $i < \theta_n$, the forest $ \mathbf{F}_i^{ \mathfrak{a}_{\mathrm{Greedy}}} $ is made of white isolated undetermined vertices, white trees of size at least $2$  rooted at a blocked vertex, and a blue tree, see Figure \ref{fig:expgreedy}. 
 \begin{figure}[!h]
 \begin{center}
 \includegraphics[width=10cm]{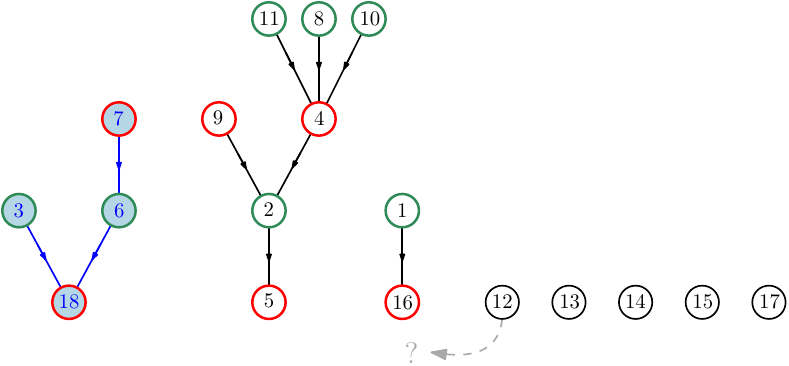}
 \caption{Illustration of a possible value of $ \mathbf{F}_{10}^{\mathfrak{a}_{ \mathrm{Greedy}}}$. As in Figure \ref{fig:transitions}, the interior color represents the color of the vertices in the peeling exploration and the boundary color represents the status of the vertices in the greedy construction of the independent set: black for the undetermined vertices, green for the active vertices and red for the blocked vertices. The new vertex to peel has label $12$ here. \label{fig:expgreedy}}
 \end{center}
 \end{figure}

A small caveat is that a priori the peeling exploration with algorithm $ \mathfrak{a}_{ \mathrm{Greedy}}$ is not defined after time $ \theta_n$: Generically, at that time, although the status in the greedy independent set of all vertices is known, the whole geometric structure of $ \mathcal{T}_n$ is not completely revealed since there are many white vertices in  $ \mathbf{F}_{ \theta_n}^{ \mathfrak{a}_{\mathrm{Greedy}}} $ (we never peel the blocked roots of white trees). For the sake of completness we may continue the peeling exploration (for example by peeling the white (blocked) root of a tree with the smallest label), but we shall not use it.

We denote by $(U_i^n, A_i^n , B_i^n)_{i \geq 0}$ the number of undetermined, active and blocked vertices in the process on a uniform Cayley tree $ \mathcal{T}_n$ with the greedy peeling algorithm $ \mathfrak{a}_{ \mathrm{Greedy}}$ i.e.\ the size of $ \widetilde{\mathcal{U}}_{i}^n$, $ \widetilde{\mathcal{A}}_{i}^n$ and $ \widetilde{\mathcal{B}}_{i}^n$. For the active and blocked vertices, we distinguish the number of white vertices $A_i^{n,w}$ and $B_i^{n,w}$ from the number of blue vertices $A_i^{n,b}$ and $B_i^{n,b}.$ Then with the notation of Theorem \ref{thm:sym}, we have $G_n = A_{ \theta_n}^n$ and using a decorated version of Proposition \ref{prop:markovtree}, we see that 
$$(U^n_i,A^{n,w}_i,B^{n,w}_i,A^{n,b}_i,B^{n,b}_i : 0 \leq i \leq \theta_n)$$
is a Markov chain. To describe its probability transitions, we introduce $ \Delta	X_i := X_{i+1} - X_i$ the increment of a random process $(X_j : j \geq 0)$ at time $i \geq 0$. Suppose that $ \widetilde{\mathcal{U}}_i^n$ is not empty. On the one hand, when  $n \in \widetilde{\mathcal{U}}_i^n$, then $n$ is the only blue vertex in $ \mathbf{F}_i^{ \mathfrak{a}_{\mathrm{Greedy}}}$, hence $A_i^{n,b} = B_i^{n,b} = 0$.
On the other hand, when $n \notin \widetilde{\mathcal{U}}_i^n$, then $A^{n,b}_{i} \geq 1$ and $B^{n,b}_i \geq 1.$
Note that as long as  $ \widetilde{\mathcal{U}}_i^n$ is neither $\{ n\}$ nor  the emptyset $ \emptyset$, the vertex $ \mathfrak{a}_{\mathrm{Greedy}} ( \mathbf{F}_{i}^{ \mathfrak{a}_{\mathrm{Greedy}}})$ is a white isolated vertex in $\mathbf{F}_{i}^{ \mathfrak{a}_{\mathrm{Greedy}}}$. Therefore, conditionally on $(U_i^n , A_i^{n,w}, B_i^{n,w} , A_i^{n,b} , B_i^{n,b})$, the transitions of the increments of $((U_i^n , A_i^{n,w}, B_i^{n,w} , A_i^{n,b} , B_i^{n,b}) : {i \geq 0})$ are given by:

If $ A_i^{n,b} = B_i^{n,b} = 0$  and $U_i^n = 1$ i.e.\ if $ \widetilde{\mathcal{U}}_{i}^n = \{ n\}$ then almost surely,  
$$ \Delta (U^n_i,A^{n,w}_i,B^{n,w}_i,A^{n,b}_i,B^{n,b}_i) = (-1, 0,0,+1,0). $$
 \begin{eqnarray}\label{table}
\begin{array}{|l||c|c|c|c|}
\multicolumn{5}{c}{ \mbox{If } A_i^{n,b} = B_i^{n,b} = 0 \mbox{ and } U_i^n \geq 2,} \\ [0.3cm]
\hline
\tiny{\mbox{with prob.}} & \scriptstyle{\frac{U_i^n -2}{n}} & \scriptstyle{\frac{A_i^{n,w}}{n}}& \scriptstyle{\frac{B_i^{n,w}}{n} }& \scriptstyle{\frac{2}{n}} \\
  \hline \hline
 \displaystyle \Delta U_i^n & -2 & -1& -1& -2 \\ \hline
 \displaystyle \Delta A^{n,w}_i& +1& 0& +1& 0\\ \hline
\displaystyle \Delta  B^{n,w}_i & +1& +1& 0& 0\\ \hline
\displaystyle \Delta A^{n,b}_i & 0& 0& 0& +1 \\ \hline
\displaystyle \Delta  B^{n,b}_i& 0&0& 0& +1\\ \hline
\end{array}
\quad
\begin{array}{|l||c|c|c|c|c|}
\multicolumn{6}{c}{ \mbox{If } A_i^{n,b} \geq 1,  B_i^{n,b} \geq 1 \mbox{ and } U_i^n \geq 1,} \\ [0.1cm]
\hline
\tiny{\mbox{with prob.}}&  \scriptstyle{ \frac{U_i^n -1}{n} } &  \scriptstyle{\frac{A_i^{n,w}}{n} }&  \scriptstyle{\frac{B_i^{n,w}}{n}} & \scriptstyle{ \frac{A_i^{n,b} ( A_i^{n,b} + B_i^{n,b} + 1)}{( A_i^{n,b} + B_i^{n,b})n}} & \scriptstyle{ \frac{B_i^{n,b} ( A_i^{n,b} + B_i^{n,b} + 1)}{( A_i^{n,b} + B_i^{n,b})n}}\\
  \hline \hline
\displaystyle \Delta U_i^n & -2 & -1& -1& -1 & -1\\ \hline
\displaystyle \Delta A^{n,w}_i& +1& 0& +1& 0& 0\\ \hline
\displaystyle \Delta  B^{n,w}_i& +1& +1& 0& 0& 0\\ \hline
\displaystyle \Delta A^{n,b}_i& 0& 0& 0& 0 & +1\\ \hline
\displaystyle \Delta  B^{n,b}_i& 0& 0& 0& +1& 0\\ \hline
\end{array}
\end{eqnarray}

\subsection{Symmetry of the greedy independent set} \label{sec:fluctuations}

We now prove Theorem \ref{thm:sym} which states the symmetry between the law of $A_{ \theta_n}^n$ the size of the maximal independent set obtained by the greedy algorithm and $ B_{ \theta_n}^n$ that of its complement. The main observation is the following lemma.

\begin{lemma} \label{lemma:sym} For $j \geq 0$, the laws of  $(A_i^{n,w}, A_i^{n,b})_{0 \leq i \leq j+1}$ and $ (B_i^{n,w} , B_i^{n,b})_{0 \leq i \leq j+1}$ are the same on the event $\ \widetilde{\mathcal{U}}_i^n \neq \{ n \}$ for all $i \leq j$.
\end{lemma}
\begin{proof} 
Notice then that the role of $(A_i^{n,w}, A_i^{n,b})$ and that of $ (B_i^{n,w} , B_i^{n,b})$ are exchangeable in the probability transitions \eqref{table} as long as $ \widetilde{\mathcal{U}}_i^n \neq \{ n\}$. The lemma follows. \end{proof}

Thanks to this lemma, we notice that conditionally on  $ \widetilde{\mathcal{U}}_{i}^n \neq \{n \}$ for all $0 \leq i \leq \theta_n$, then the size $A_{\theta_n}^n$ of the greedy independent set has the same law as that its complement $B_{\theta_n}^n$. Otherwise, if there exists $i \geq 0$ such that $\widetilde{\mathcal{U}}_{i}^n = \{n \}$, then this $i$ corresponds to $\theta_n - 1$ and  in that case, $A_{\theta_n-1}^n$ and $B_{\theta_n-1}^n$ have the same law and $A_{\theta_n}^n = G_n = A_{\theta_n-1}^n +1$ and $B_{\theta_n}^n = B_{\theta_n-1}^n$. Hence, the random variable $ \mathcal{E}_n$ in Theorem \ref{thm:sym} is the indicator function of the event $ \{ \exists i \geq 0, \ \widetilde{\mathcal{U}}_{i}^n = \{n \} \}$ and to prove Theorem \ref{thm:sym}, it only remains to show that 
$$\mathbb{P} ( \mathcal{E}_n = 1) = \mathbb{P} ( \exists i, \ \widetilde{\mathcal{U}}_{i}^n = \{n \} ) \xrightarrow[n\to\infty]{\ }  1/4.$$

\noindent To compute the above probability and also obtain the fluctuations of $G_n$, we study the \emph{fluid limit} of our system (see \cite[Chapter 11]{ethier2009markov}). The transition probabilities \eqref{table} in both cases point to the idea of gathering blue and white vertices and focusing only on their status. Indeed,  the Markov chain  $(U_i^n , A_i^{n}, B_i^{n} : i \geq 0)$ has bounded increments and for $i < \theta_n$, 
\begin{equation} \label{eq:deff}
 \left| \mathbb{E}  \left[ \begin{pmatrix} \Delta U_i^n \\ \Delta A_i^n \\ \Delta B_i^n \end{pmatrix} \middle| \begin{pmatrix} U_i^n \\  A_i^n \\ B_i^n \end{pmatrix} \right]  - F \left(\frac{U_i^n}{n}, \frac{A_i^n}{n}, \frac{B_i^n}{n}\right) \right|  \leq \frac{3}{n}, \quad \mbox{where } F (x,y,z) = \begin{pmatrix} -2x -y-z \\ x+z \\ x+y \end{pmatrix}.
\end{equation}
\noindent This suggests to study the deterministic system which is solution of the following equations:
\begin{equation*}
\left\lbrace 
\begin{array}{ll}
u'(t) = -2 u(t) -a(t)-b(t), \qquad \qquad & u(0) = 1, \\
a'(t) = u(t) + b(t), & a(0) = 0, \\
b'(t) = u(t) + a(t), & b(0) = 0.
\end{array}\right.
\end{equation*}
The solution is given by 
\begin{equation*}
\left\lbrace 
\begin{array}{l}
u(t) = 2 e^{-t}- 1, \\
a(t) = b(t) = 1-e^{-t}.
\end{array}\right.
\end{equation*}
In particular $t^* = \min \{ t \geq 0, \ u(t) = 0 \} = \ln (2)$ and $a (t^* ) =1/2$. We can now apply \cite[Theorem 2.1 p456]{ethier2009markov} to the process $ X_n (t) = 1/n  \cdot (U^n_{ \lfloor tn \rfloor}, A^n_{ \lfloor tn \rfloor},B^n_{ \lfloor tn \rfloor})$ which starts at $X_n (0) = (1, 0, 0)$ with, using the notation of \cite{ethier2009markov}, the function $F$ given by \eqref{eq:deff} and 
$$\left\lbrace 
\begin{array}{l}
 \beta_{(-2,1,1)} (x,y,z) = x, \\
 \beta_{(-1,0,1)} (x,y,z) = y, \\
\beta_{(-1,1,0)} (x,y,z) = z.
\end{array}\right. $$
We obtain that at least for all $t<t^*$, for any $ \varepsilon > 0$, 
\begin{equation*}
\mathbb{P} \left( \sup_{0 \leq s \leq t} \left|  \left(\frac{U_{ \lfloor sn \rfloor}^n}{n} , \frac{A^{n}_{ \lfloor sn \rfloor}}{n}, \frac{B^{n}_{ \lfloor sn \rfloor}}{n}\right) - (u(s),a(s),b(s)) \right| > \varepsilon \right) \xrightarrow[n\to\infty]{\ } 0.
\end{equation*}
Since $u'(t^*) < 0$, we can apply \cite[Theorem 4.1 p464]{ethier2009markov} which states that the stopping time of our Markov chain also concentrates around its expected value and has Gaussian fluctuations around it:
\begin{equation}\label{eq:cvtn}
\sqrt{n} \left( \frac{\theta_n}{n} - t^*\right) \xrightarrow[n\to\infty]{(d)}\mathcal{N} (0, 3/4). 
\end{equation}
Lastly, notice that $ \mathcal{E}_n = 1$ when the vertex $n$ stays undetermined at each step $1 \leq i \leq \theta_n- 1$, that is, with probability $1 - 2/n $ at each step, and by the the previous equation $(1 - 2/n)^{ \theta_n}$ converges in probability towards $e^{-2 t^*}$ and is bounded by $1$. Hence we obtain $$ \mathbb{P} ( \mathcal{E}_n = 1) = \mathbb{E} \left[ \left( 1 - \frac{2}{n}\right)^{\theta_n - 1} \right] \xrightarrow[n\to\infty]{ \ } \mathrm{e}^{-2 t^*} = \frac{1}{4}, $$
which concludes the proof of Theorem \ref{thm:sym}. 
%
%

 \bibliographystyle{siam}
\bibliography{/Users/contat/Dropbox/Articles/biblio}

\end{document}